\documentclass{amsart}

\usepackage[T1]{fontenc}
\usepackage[utf8]{inputenc}
\usepackage{
latexsym, 
amssymb, 
amsmath, 
graphicx, 
multicol,
listings,
algorithm,
algorithmic
}

\newtheorem{thm}{Theorem}[section]
\newtheorem{prop}[thm]{Proposition}
\newtheorem{lem}[thm]{Lemma}

\theoremstyle{definition}
\newtheorem{defi}[thm]{Definition}
\newtheorem{rem}[thm]{Remark}

\def\w{{\bf w}}
\def\T{{\mathcal T}}

\def\B{{\mathcal B}}
\def\Z{{\mathbb Z}}

\def\scaleFigure{0.7}

\title{The tree structure in staircase tableaux}

\author{Jean-Christophe Aval}
\address{LaBRI, Université de Bordeaux, 351 cours de la Libération, 33405 Talence, France}
\author{Adrien Boussicault}
\address{LaBRI, Université de Bordeaux, 351 cours de la Libération, 33405 Talence, France}
\thanks{The first two authors are supported by the ANR (PSYCO project -- ANR-11-JS02-001)}
\author{Sandrine Dasse-Hartaut}
\address{LIAFA, Université Paris Diderot - Paris 7, Case 7014 75205 Paris Cedex 13, France}
\thanks{The third  author is supported by the  ANR (IComb project)}

\date{\today}

\begin{document}
\maketitle

\begin{abstract}
Staircase tableaux are combinatorial objects which appear as key tools in the study
of the PASEP physical model.
The aim of this work is to show how the discovery of a tree structure in staircase tableaux
is a significant feature to derive properties on these objects.
\end{abstract}

\section*{Introduction}

The {\em Partially Asymmetric Simple Exclusion Process} (PASEP) is a physical model 
in which $N$ sites on a one-dimensional lattice are either empty
or occupied by a single particle. These particles may hop to the
left or to the right with fixed probabilities, which defines
a Markov chain on the $2^N$ states of the model.
The explicit description of the stationary probability of the PASEP
was obtained through the Matrix-Ansatz \cite{derrida}.
Since then, the links between specializations of this model 
and combinatorics have been the subject of an important 
research (see for example \cite{DS,CW1,CSSW}).
A great achievment is the description of the stationary distribution
of the most general PASEP model through statistics defined on 
combinatorial objects called staircase tableaux \cite{CW2}.

The objective of our work is to show how we can reveal an underlying 
tree structure in staircase tableaux and use it to obtain combinatorial properties.
In this paper, we focus on some specializations of the  {\em fugacity partition function} 
$Z_n(y;\alpha,\beta,\gamma,\delta;q)$, defined in the next section 
as a $y$-analogue of the classical partition function of staircase tableaux.

The present paper is divided into two sections.
Section \ref{sec:TLST} is devoted to the definition of labeled tree-like tableaux,
which are a new presentation of staircase tableaux,
and to the presentation of a tree structure and of an insertion algorithm on these objects.
Section \ref{sec:app} presents combinatorial applications of these tools.
We get:
\begin{itemize}
\item a new and natural proof of the formula for $Z_n(1; \alpha,\beta,\gamma,\delta ; 1)$;
\item a study of $Z_n(y;1,1,0,1;1)$; 
\item a bijective proof for the formula of $Z_n(1;1,1,1,0;0)$.
\end{itemize}

\section{Staircase tableaux and labeled tree-like tableaux}
\label{sec:TLST}

\subsection{Staircase tableaux}

\begin{defi}
A {\em staircase tableau} $\T$ of size $n$ is a Ferrers diagram of “staircase”
shape $(n, n - 1, \dots, 2, 1)$ such that boxes are either empty or labeled with $\alpha$, $\beta$, 
$\gamma$, or
$\delta$, and satisfying the following conditions:
\begin{itemize}
\item no box along the diagonal of $\T$ is empty;
\item all boxes in the same row and to the left of a $\beta$ or a $\delta$ are empty;
\item all boxes in the same column and above an $\alpha$ or a $\gamma$ are empty.
\end{itemize}
\end{defi}

Figure \ref{fig:ST} (left) presents a staircase tableau $\T$ of size 5.

\begin{figure}[ht]
    \centerline{
		\includegraphics[scale=\scaleFigure]{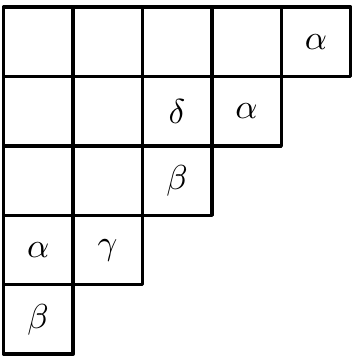}
		\hspace{1cm}
    	\includegraphics[scale=\scaleFigure]{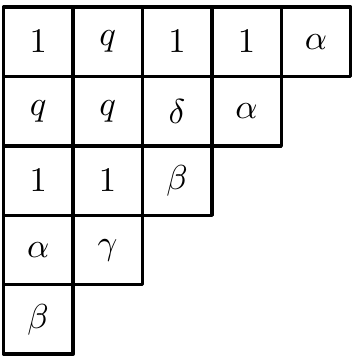}
	}
\caption{A staircase tableau and its weight labels}
\label{fig:ST}
\end{figure}

\begin{defi}\label{weight}
The {\em weight} $\w(\T)$ of a staircase tableau $\T$ is a monomial in
$\alpha, \beta, \gamma, \delta$ and $q$,  which we obtain as follows.
Every blank box of $\T$ is assigned a $q$ or a $1$, based on the label of the closest
labeled box to its right in the same row and the label of the closest labeled box
below it in the same column, such that:
\begin{itemize}
\item every blank box which sees a $\beta$ to its right gets a $1$;
\item every blank box which sees a $\delta$ to its right gets  a $q$;
\item every blank box which sees an $\alpha$ or $\gamma$ to its right,
  and an $\alpha$ or $\delta$ below it, gets  a $1$;
\item every blank box which sees an $\alpha$ or $\gamma$ to its right,
  and a $\beta$ or $\gamma$ below it, gets a $q$.
\end{itemize}
After filling all blank boxes,
we define $\w(\T)$
to be the product of all labels in all boxes.
\end{defi}

The weight of the staircase tableau $T$ on Figure~ \ref{fig:ST} is $\w(\T)=q^3 \alpha^3 \beta^2 \gamma \delta$.

There is a simple correspondence between the states of the PASEP 
and the diagonal labels of staircase tableaux: diagonal boxes
may be seen as sites of the model, and $\alpha$ and $\delta$ 
(resp. $\beta$ and $\gamma$)
diagonal labels correspond to occupied (resp. unoccupied) sites.
We shall use a variable $y$ 
to keep track of the number of particles
in each state.
To this way, we define $t(\T)$ to be the number of labels $\alpha$ or $\delta$ along the diagonal of $\T$.
For example  the tableau $\T$ in Figure~ \ref{fig:ST} has $t(\T)=2$. 
The {\em fugacity partition function}
of the PASEP is defined as 
$$Z_n(y;\alpha,\beta,\gamma,\delta;q)=
\sum_{\T\ {\rm of } \ {\rm size}\ n}\w(\T)y^{t(\T)}.
$$

\subsection{Labeled tree-like tableaux}

We shall now define another class of objects, called labeled tree-like tableaux.
They appear as a labeled version of tree-like tableaux (TLTs) defined in \cite{TLT}.
These tableaux are in bijection with staircase tableaux, and present two nice properties
inherited from TLTs: an underlying tree structure, and an insertion algorithm which provides 
a useful recursive presentation.

In a Ferrers diagram $D$, 
the {\em border edges} are the edges that stand at the end of rows or columns.
The number of border edges is clearly the half-perimeter of $D$.
For any box $c$ of $D$, we define $L_D(c)$ as the set of boxes placed in the same column and above $c$ in $D$,
and  $A_D(c)$ as the set of boxes placed in the same row and to the left of $c$ in $D$.
By a slight abuse, we shall use the same notations for any tableau $T$ of shape $D$.
These notions are illustrated at Figure~ \ref{fig:bordedges-legarm}.

\begin{figure}[ht]
    \centerline{
		$
		\begin{array}{c}
			\includegraphics[scale=\scaleFigure]{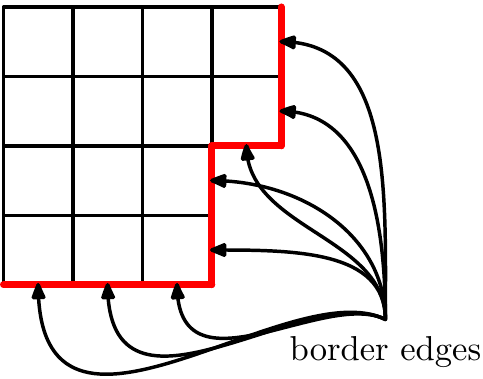}
		\end{array}
		$
			\hspace{1cm}
		$
		\begin{array}{c}
			\includegraphics[scale=\scaleFigure]{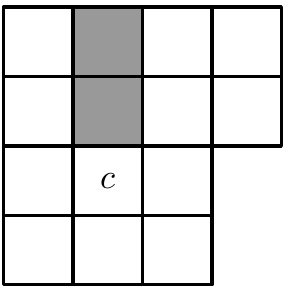}
		\end{array}
		$
		\hspace{1cm}
		$
		\begin{array}{c}
    		\includegraphics[scale=\scaleFigure]{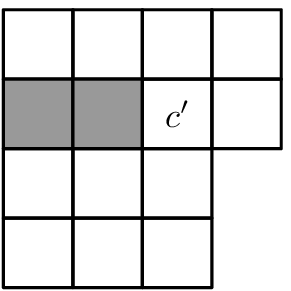}
		\end{array}
	$
	}
\caption{Border edges, $L_D(c)$ and $A_D(c')$}
\label{fig:bordedges-legarm}
\end{figure}

\begin{defi}\label{def:LTLT}
A {\em labeled tree-like tableau} (LTLT) $T$
of size $n$ is a Ferrers diagram of half-perimeter $n+1$
such that some boxes and all border edges are labeled with $1$, $\alpha$, $\beta$, 
$\gamma$, or
$\delta$, and satisfying the following conditions:
\begin{itemize}
\item the Northwestern-most box (root box) is labeled by $1$;
\item the labels in the first row and the first column are the only labels $1$; 
\item in each row and column, there is at least one labeled box;
\item for each box $c$ labeled by $\alpha$ or $\gamma$, all boxes in $L_T(c)$ are empty and at least one box in $A_T(c)$ is labeled;
\item for each box $c$ labeled by $\beta$ or $\delta$, all boxes in $A_T(c)$ are empty and at least one box in $L_T(c)$ is labeled.
\end{itemize}
\end{defi}

\begin{prop}\label{prop:bij}
For $n\ge2$, LTLTs of size $n$ are in bijection with staircase tableaux of size $n-1$.
\end{prop}
\begin{proof}
Let us describe a correspondence $\Lambda$ that sends a staircase tableau $\T$ of size $n-1$
to an LTLT $T$. It consists in the following steps (see Figure~ \ref{fig:L-bij}):
\begin{itemize}
\item we add to $\T$ a hook $(n+1,1^{n})$;
\item in this hook, we label by $1$: the root-cell, the border edges and the cells in the first row that see a $\beta$ or $\delta$ below, 
and the cells in the first column that see an $\alpha$ or $\gamma$ to their right;
\item for each label $\alpha$ or $\gamma$ (resp. $\beta$ or $\delta$) on the diagonal, we erase the corresponding column (resp. row) in the tableau,
which has to be empty, and put the label on the vertical (resp. horizontal) border edge to its left (resp. above it).
\end{itemize}
The result of these operations is an LTLT $\Lambda(\T)$ of size $n$.

It is straightforward to construct the inverse of $\Lambda$, since each operation may be reversed,
thus proving that $\Lambda$ is a bijection.
\end{proof}

\begin{figure}[ht]
    \centerline{
		$
		\begin{array}{c}
			\includegraphics[scale=\scaleFigure]{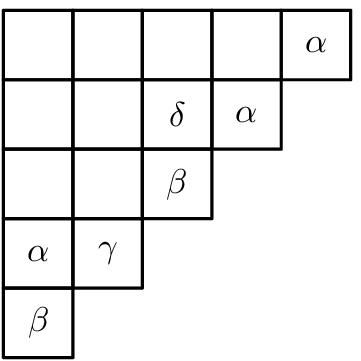}
		\end{array}
		\longleftrightarrow
		\begin{array}{c}
			\includegraphics[scale=\scaleFigure]{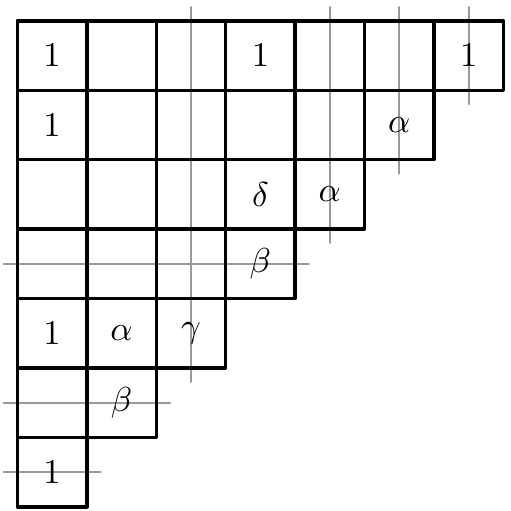}
		\end{array}
		\longleftrightarrow
		\begin{array}{c}
			\includegraphics[scale=\scaleFigure]{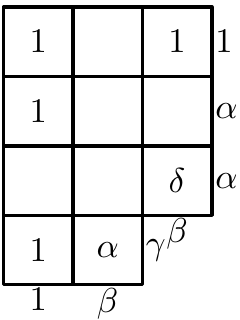}
		\end{array}
		$
	}
\caption{The bijection $\Lambda$}
\label{fig:L-bij}
\end{figure}

A nice feature of the notion of LTLT is its underlying tree structure. 
Let us consider an LTLT $T$ of size $n$.
We may see each label of $T$ as a node, 
and Definition~ \ref{def:LTLT} ensures that each node
(except the NE-most one, which appears as the root)
has either a node above it or to its left, which may be seen as its father.
We refer to Figure~ \ref{fig:tree} which illustrates this property. 
A {\em crossing} is a box $c$ such that 
\begin{itemize}
\item there is a label to the left and to the right of $c$;
\item there is a label above and below $c$.
\end{itemize}
In this way, we get a labeled binary tree with some additional information: crossings of edges.
If we forget the crossings, we have a binary tree in which each internal node and leaf is labeled.
As a consequence, any LTLT is endowed with an underlying binary tree structure,
such that the size of the LTLT is equal to the number of internal nodes in its underlying binary tree.

\begin{figure}[ht]
    \centerline{
		$
		\begin{array}{c}
			\includegraphics[scale=\scaleFigure]{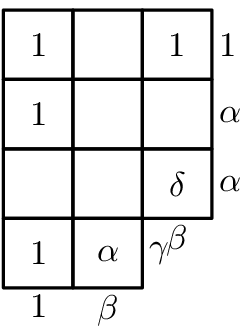}
		\end{array}
		$
		\hspace{1cm}
		$
		\begin{array}{c}
			\includegraphics[scale=\scaleFigure]{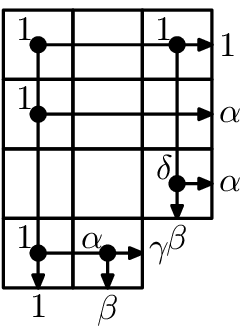}
		\end{array}
		$
		\hspace{1cm}
		$
		\begin{array}{c}
			\includegraphics[scale=\scaleFigure]{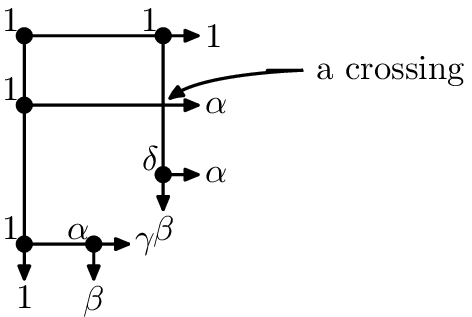}
		\end{array}
		$
	}
	\caption{The tree structure of staircase tableaux}
	\label{fig:tree}
\end{figure}

\subsection{The insertion algorithm}

Given an LTLT $T$ and a border edge $e$, a {\em compatible bi-label} is the choice of a (new) couple $(x,y)$ of labels such that
\begin{itemize}
\item if $e$ is in the first row of $T$ (thus vertical): $x=1$ and $y\in\{\beta,\delta\}$; 
\item if $e$ is in the first column of $T$ (thus horizontal): $y=1$ and $x\in\{\alpha,\gamma\}$; 
\item otherwise: $x\in\{\alpha,\gamma\}$ and $y\in\{\beta,\delta\}$.
\end{itemize}

\begin{defi}\label{def:col_add}
Given an LTLT $T$, a vertical border edge $e$ (thus at the end of a row $r$) with label $\ell$, 
and a compatible bi-label $(x,y)$,
the {\em column addition} in $T$ at edge $e$, with new label $(x,y)$ 
is defined as follows:
\begin{itemize}
\item we add a cell to $r$ and to all rows above it; 
\item since vertical and horizontal border edges on the right of $e$ are shifted horizontally, we shift also the corresponding labels;
\item we label the new box in row $r$ by $\ell$;
\item we label the two new vertical and horizontal border edges respectively by $x$ and $y$.
\end{itemize}
If the edge $e$ is horizontal, we define in the same way the {\em row addition}.
\end{defi}

\begin{figure}[ht]
    \centerline{
		$
		\begin{array}{c}
			\includegraphics[scale=\scaleFigure]{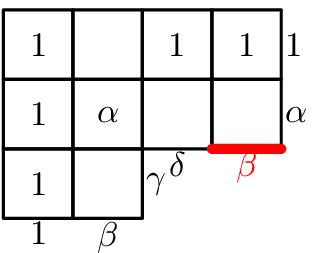}
		\end{array}
		\xrightarrow[(x,y) = (\gamma,\beta)]{}
		\begin{array}{c}
			\includegraphics[scale=\scaleFigure]{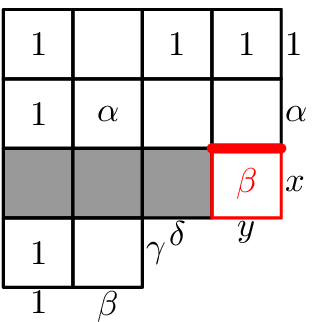}
		\end{array}
		\xrightarrow[\text{substitution}]{\text{after $x$ and $y$}}
		\begin{array}{c}
			\includegraphics[scale=\scaleFigure]{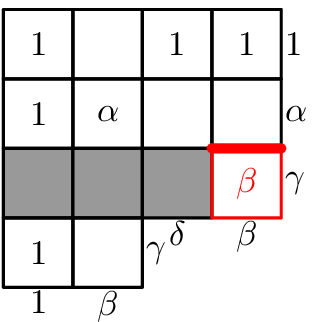}
		\end{array}
		$
	}
\caption{The row addition}
\label{fig:horisontal_insertion}
\end{figure}

Given two Ferrers diagrams $D_1\subseteq D_2$, we say that the set of cells $S=D_2-D_1$ (set-theoretic difference) 
is a {\em ribbon} if it is connected (with respect to adjacency) and contains no $2\times 2$ square. 
In this case we say that $S$ can be added to $D_1$, or that it can be removed from $D_2$.
For our purpose, we shall only consider the addition of a ribbon to an LTLT $T$
between a vertical border edge $e_1$ and an horizontal border edge $e_2$.
As in the row/column insertion, we observe that vertical (resp. horizontal) border edges are shifted
horizontally (resp. vertically), thus we shift also the corresponding labels.
Figure \ref{fig:ribbon_insertion} illustrates this operation.

\begin{defi}\label{def:spec}
Let $T$ be an LTLT. The {\em special box} of $T$ is the Northeast-most labeled box among those that occur at the bottom of a column.
This is well-defined since the bottom row of $T$ contains necessarily a labeled box.
\end{defi}

\begin{minipage}{0.6\columnwidth}
\begin{figure}[H]
    \centerline{
		$
		\begin{array}{c}
			\includegraphics[scale=\scaleFigure]{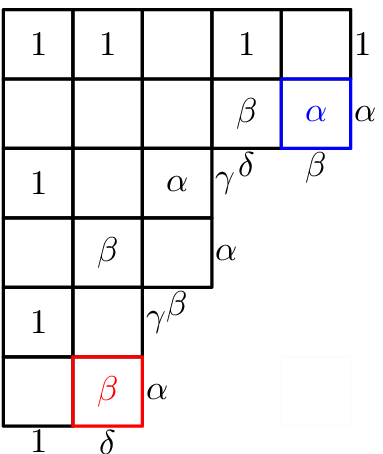}
		\end{array}
		\xrightarrow[insertion]{ribbon}
		\begin{array}{c}
			\includegraphics[scale=\scaleFigure]{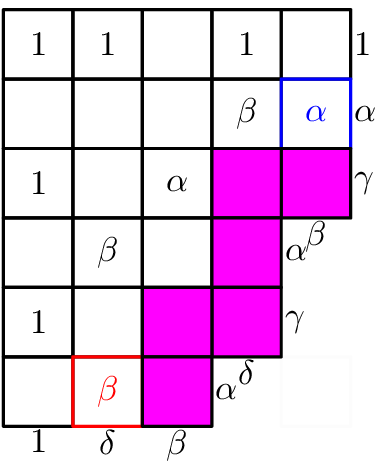}
		\end{array}
		$
	}
\caption{The ribbon insertion}
\label{fig:ribbon_insertion}
\end{figure}
\end{minipage}
\begin{minipage}{0.4\columnwidth}
\begin{figure}[H]
    \centerline{
		$
		\begin{array}{c}
			\includegraphics[scale=\scaleFigure]{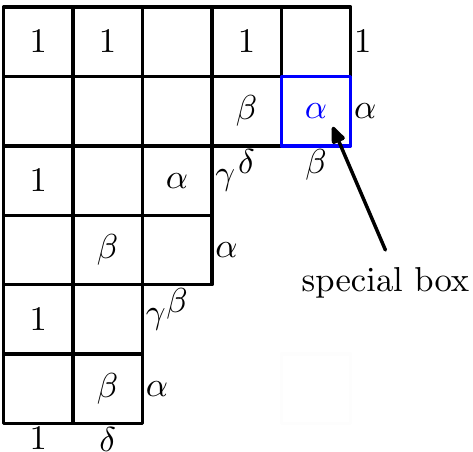}
		\end{array}
		$
	}
\caption{The special box}
\label{fig:special_box}
\end{figure}
\end{minipage}

\begin{algorithm}[ht]
\caption{Insertion procedure}
\label{algo:insertion}
\begin{algorithmic}[1]
\REQUIRE an LTLT $T$ of size $n$ together with the choice of one of its border edges $e$, 
and a compatible bi-label $(x,y)$.
\STATE \label{etape_reperer_case_speciale} Find the special box $s$ of $T$.
\STATE \label{etape_inserer_colonne} Add a row/column to $T$ at edge $e$ with new bi-label $(x,y)$.
\STATE \label{etape_ajouter_ruban} If $e$ is to the left of $s$, perform a  ribbon addition between $e$ and $s$.
\ENSURE a final  LTLT $T'$ of size $n+1$.
\end{algorithmic}
\end{algorithm}

\begin{figure}[ht]
    \centerline{
		$
		\begin{array}{c}
			\includegraphics[scale=\scaleFigure]{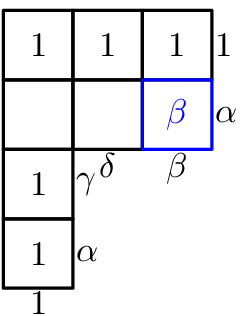}
		\end{array}
		\longrightarrow
		\begin{array}{c}
			\includegraphics[scale=\scaleFigure]{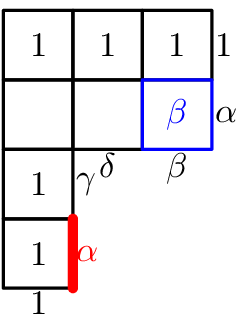}
		\end{array}
		\longrightarrow
		\begin{array}{c}
			\includegraphics[scale=\scaleFigure]{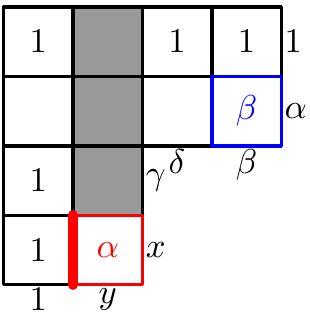}
		\end{array}
		\longrightarrow
		\begin{array}{c}
			\includegraphics[scale=\scaleFigure]{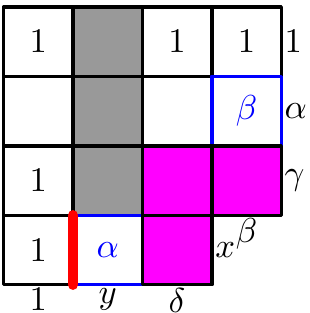}
		\end{array}
		$
	}
\caption{The insertion procedure}
\label{fig:insertion_procedure}
\end{figure}

\begin{prop}\label{prop:insertion}
The insertion algorithm \ref{algo:insertion} induces a bijection between 
\begin{itemize}
\item LTLTs of size $n$ together with the choice of a border edge and a compatible bi-label, 
\item LTLTs of size $n+1$.
\end{itemize}
\end{prop}

\begin{proof}
We shall only give the key ingredient, which is analog to Theorem 2.2 in \cite{TLT}: the insertion algorithm is constructed in such a way
that the added labeled box becomes the special box of the new LTLT.
This implies that the process may be inversed, thus proving Proposition~ \ref{prop:insertion}.
\end{proof}

\begin{rem}
An important feature of this algorithm is that it provides a recursive comprehension
focused on the border edges of the LTLT, {\it i.e.} on the diagonal element of the staircase tableaux.
Up to now, the recursive approach was with respect to the first column, which is of course less 
significant since the state of the PASEP is encoded by the labels on the diagonal.
We give in the next section three examples of the use we can make of the tree structure and the insertion  algorithm.
\end{rem}

\section{Combinatorial applications}
\label{sec:app}

\subsection{A product formula for $Z_n(1; \alpha,\beta,\gamma,\delta ; 1) $}

\begin{prop}
\label{prop:gen_fct}
When $q = y = 1$,
\begin{equation}\label{eq:gen_fct}
Z_n(1 ; \alpha,\beta,\gamma,\delta ; 1) = \prod_{j=0}^{n-1}(\alpha+\beta+\gamma+\delta+j(\alpha+\gamma)(\beta+\delta)).
\end{equation}
\end{prop}

\begin{proof}
Proposition \ref{prop:insertion} implies that given a LTLT $T$ of size $n$, we can build through the insertion algorithm
exactly $2+4(n-1)+2=4n$ different LTLTs $T'$ of size $(n+1)$.
Using  Proposition~ \ref{prop:bij}, let $\T=\Lambda^{-1}(T)$.
The contribution to $Z_n(1 ; \alpha,\beta,\gamma,\delta ; 1) $ of the $4n$ staircase tableaux $\Lambda^{-1}(T')$
is precisely 
$$\w(\T)\, ((\alpha+\gamma)+(n-1)(\alpha+\gamma)(\beta+\delta)+(\beta+\delta)),$$
whence (\ref{eq:gen_fct}).
\end{proof}

\begin{rem}
Proposition~ \ref{prop:gen_fct} corresponds to Theorem 4.1 in \cite{CSSW}. 
The insertion algorithm gives a trivial and natural explanation for this formula.
Moreover, Proposition~ \ref{prop:gen_fct} implies that the number of staircase tableaux of size $n$
is given by $4^n\, n!$.
It is clear that we may use the insertion algorithm to build a recursive bijection between staircase tableaux of size $n$
and objects enumerated by $4^nn!$ such as doubly signed permutations, 
{\em i.e.} triple $(\sigma,\varepsilon_1,\varepsilon_2)$ where 
$\sigma$ in a permutation of $n$ and $\varepsilon_1$, $\varepsilon_2$ 
two vectors of $(-1,+1)^n$.

\end{rem}

\subsection{Study of $Z_n(y;1,1,0,1;1)$}

In this part, we consider staircase tableaux without any $\gamma$ label. 
We denote by $\T(n,k)$ the number of such tableaux $\T$ of size $n$ with $k$ labels $\alpha$ or $\delta$ in the diagonal, 
{\it i.e.} such that $t(\T)=k$.

Our goal is to get a recursive formula on $\T(n,k)$ numbers.
We consider a staircase tableau $\T$ of size $n-1$ with  $t(\T)=k$, and we let $T=\Lambda(\T)$ its associated LTLT of size $n$.
Now we examine the possible insertions on $T$ to obtain an LTLT $T'$ of size $n+1$: let $e$ be the edge where the insertion occurs, 
$z$ be the label of $e$ (which does not appear on the border edges of $T'$),
and $(x,y)$ the compatible bi-label which appears on the border edges of $T'$.
We denote by $\T'$ the staircase tableau $\Lambda ^{-1}(T')$ of size $n$.

We get that $t(\T')-t(\T)$ can be equal to (we recall that staircase tableaux of size $n$ are in bijection with LTLT of size $n+1$)
\begin{itemize}
\item $0$ when $(z,x,y)$ is either $(\alpha,\alpha,\beta)$ or $(\delta,\alpha,\beta)$ or $(1,1,\beta)$, which gives $k+1$ possibilities;
\item $1$ when $(x,y,z)$ is either $(\alpha,\alpha,\delta)$ or $(\delta,\alpha,\delta)$ or $ (\beta,\alpha,\beta)$ or $(1,1,\delta)$ or $(1,\alpha,1)$, which gives $n+1$ possibilities; 
\item $2$ when $(x,y,z)=(\beta,\alpha,\delta)$ which gives $n-1-k$ possibilities (the number of border edges labeled by $\beta$).
\end{itemize}

Putting all this together, we get the following recursive formula (for all $n \ge 0$ and $k \in \mathbb{Z}$)
\begin{equation}\label{eq:stable}
\T(n,k)=(k+1)\T(n-1,k)+(n+1)\T(n-1,k-1)+(n-k+1)\T(n-1,k-2)
\end{equation}
from which we deduce a property of the polynomial $Z_n(y;1,1,0,1;1)$.

\begin{prop}\label{prop:stable}
The polynomial $Z_n(y;1,1,0,1;1)$ has all its roots in the segment $]-1,0[$.
As a  consequence, it is stable and log-concave.
\end{prop}
\begin{proof}
We let $P_n(y)=\sum\limits_{k \ge 0}\T(n,k)y^k=Z_n(y;1,1,0,1;1)$
and use (\ref{eq:stable}) to write
\begin{eqnarray*}
P_n(y)&=&\sum_k(k+1)\T(n-1,k)y^k+\sum_k(n+1)\T(n-1,k-1)y^k\\
            &+&\sum_k(n-k+1)\T(n-1,k-2)y^k\\
           &=&(1+(n+1)y+(n-1)y^2)P_{n-1}(y)+(y-y^3)P_{n-1}'(y)\\
	  &=&\big((y-y^3)e^{\lambda_n(y)}P_{n-1}(y)\big)' \times e^{-\lambda_n(y)}
\end{eqnarray*}
with $\lambda_n(y)=-(n+\frac{3}{2})\ln (1-y) -\frac12\ln(1+y)$. Then we check that the equality
$$P_n(y)=f_n'(y)e^{-\lambda_n(y)}$$
with $f_n(y)=(y^3-y)e^{\lambda_n(y)}P_{n-1}(y)$ implies by induction on $n$ that $P_n$ has at least $n$ distinct zeros in $]-1,0[$.
Since $n$ is the degree of $P_n$, the conclusion follows.
\end{proof}

\begin{rem}
Equation (\ref{eq:stable}), as well as Proposition~ \ref{prop:stable} are new.
Our insertion algorithm shows its strength when it comes to study recursively the diagonal 
in staircase tableaux, which is meaningful in the PASEP model: it is by far more natural
than the already studied \cite{CW2} recursion with respect to the first column of the tableau.
\end{rem}

\subsection{A bijective proof for $Z_n(1;1,1,1,0;0)$.}

We show how our tools lead to a bijection between staircase tableaux without any label $\delta$ or weight $q$ 
and a certain class of paths enumerated by the sequence {\tt A026671} of  \cite{oeis}.
This is an answer to Problem 5.8 of \cite{CSSW}.

A  {\em lazy path} of size $n$ is a path on the lattice $\Z\times\Z$
\begin{itemize}
\item starting at $(0,0)$, ending at $(2n,0)$,
\item whose steps are  $(1,1)$, $(1,-1)$ or $(2,0)$,
\item such that its steps $(2,0)$ only appear on the axis $(0,x)$.
\end{itemize}
Figure \ref{fig:lazy_path} is an example of a lazy path.

\begin{figure}[ht]
    \centerline{
		$
		\begin{array}{c}
			\includegraphics[scale=\scaleFigure]{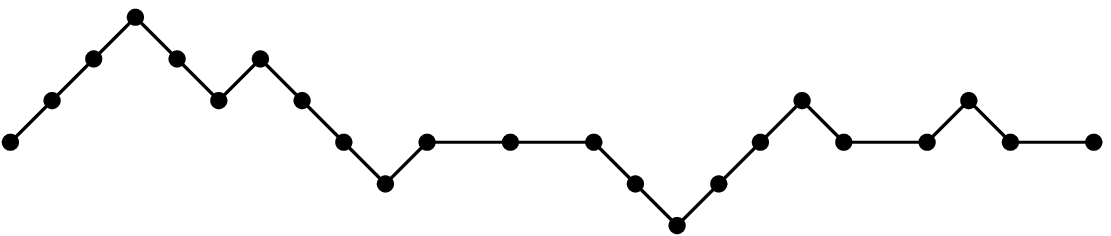}
		\end{array}
		$
	}
\caption{A lazy path of size $13$}
\label{fig:lazy_path}
\end{figure}

Let us consider LTLTs which are in bijection with staircase tableaux without $\delta$ or $q$.
The first observation is that these tableaux correspond to trees without any crossing, since a crossing sees an $\alpha$ or a  $\gamma$ 
to its right and a $\beta$ below, which gives a weight $q$
({\it cf.} Figure~ \ref{fig:no_cross}).
\begin{figure}[ht]
    \centerline{
		$
		\begin{array}{c}
			\includegraphics[scale=\scaleFigure]{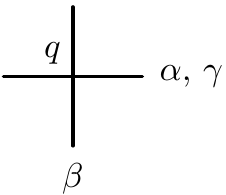}
		\end{array}
		$
	}
\caption{A crossing in an LTLT gives a weight $q$}
\label{fig:no_cross}
\end{figure}
Thus we have to deal with binary trees whose left sons only have one choice of label ($\beta$)
and whose right sons may have one ($\alpha$) or two choices ($\alpha$ or $\gamma$) of labels.
The restriction of no weight $q$ is equivalent to forbidding any point in the tree 
which sees an $\alpha$ or a  $\gamma$  
to its right and a $\gamma$ below it.
Since we deal with binary trees (without any crossing), we get that the only nodes $v$ where we may 
put a label $\gamma$ are such that the path in the tree from the root to $v$ contains exactly one left son,
followed by a right son.
These nodes are illustrated on Figure~ \ref{fig:pos_gamma}.
\begin{figure}[ht]
    \centerline{
		$
		\begin{array}{c}
			\includegraphics[scale=\scaleFigure]{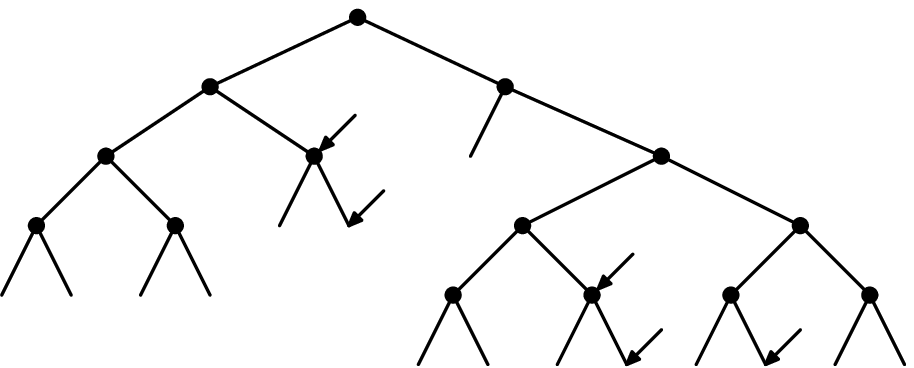}
		\end{array}
		$
	}
\caption{Arrows give the possible positions for a $\gamma$ label}
\label{fig:pos_gamma}
\end{figure}
We may shift these labels to their father in the tree, and define the {\em left depth} of a node $v$ in a binary tree
as the number of left sons in the path from the root to $v$, and using the bijection $\Lambda$, we get the following statement.

\begin{lem}\label{lem:bij}
The set of staircase tableaux of size $n$ without any $\delta$ label or $q$ weight
is in bijection with the set $\B_n$ of binary trees of size $n$ whose nodes of left depth equal to $1$
are labeled by $\alpha$ or $\gamma$.
\end{lem}

We shall now code the trees in $\B_n$ by lattice paths.
To do this, we use a deformation of the classical bijection \cite{stanley} between binary trees 
and Dyck paths: we go around the tree, starting at the root and omitting the last external node,
and we add to the path a step  $(1,1)$ when visiting (for the first time) an internal node,
or a step $(1,-1)$ when visiting an external node.
Let us denote by $\pi(T)$ the  (Dyck) path associated to the binary tree $T$ under this procedure.
It is well-known that $\pi$ is a bijection between binary trees with $n$  internal nodes
and Dyck paths of length $2n$ (of size $n$).

If we use the same coding, but omitting the root and the last $2$ external nodes,
we get a bijection between binary trees with $n$ internal nodes and {\em almost-Dyck} paths
(whose ordinate is always $\ge -1$) of size $n-1$. 
In the sequel, we shall call {\em factor} of a path a minimal sub-path starting from the axis and ending on the axis.
We may replace the negative factors by steps  $(0,2)$ to get a  bijection $\pi'$ between  binary trees of size
$n$ and {\em positive} lazy paths of size $n-1$ (these objects appear under the name (${\rm x}^4$) in \cite{catadd}).
Figure \ref{fig:bij_pi} illustrates bijections $\pi$ and $\pi'$. 
\begin{figure}[ht]
    \centerline{
		$
		\begin{array}{c}
			\includegraphics[scale=\scaleFigure]{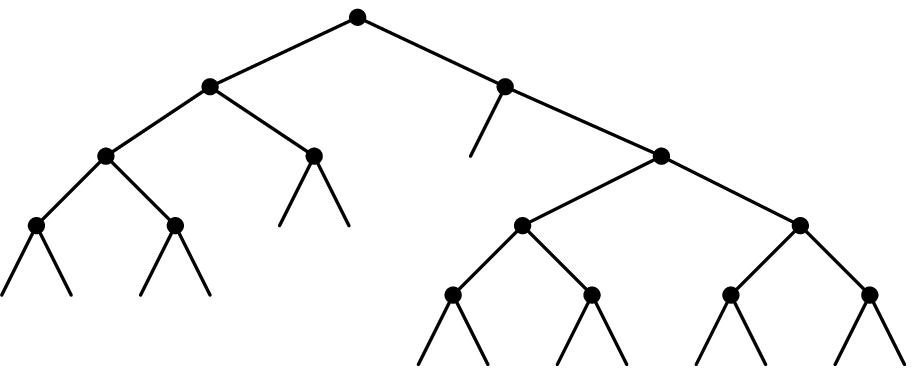}
		\end{array}
		$
	}
    \centerline{
		$
		\begin{array}{c}
			\includegraphics[scale=\scaleFigure]{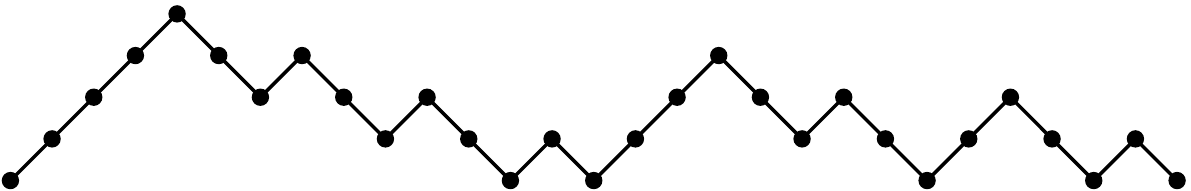}
		\end{array}
		$
	}
    \centerline{
		$
		\begin{array}{c}
			\includegraphics[scale=\scaleFigure]{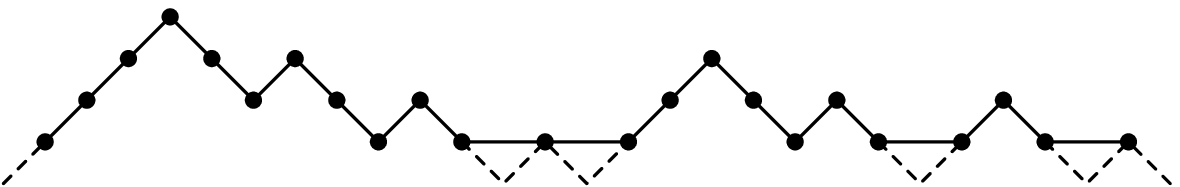}
		\end{array}
		$
	}
\caption{The bijections $\pi$ and $\pi'$}
\label{fig:bij_pi}
\end{figure}

We observe that the nodes with left depth equal to $1$ in a binary tree $B$ correspond to steps 
$(1,1)$ which start on the axis in  $\pi'(B)$, thus to strictly positive factors in $\pi'(B)$. 
These nodes may be labeled with $\alpha$ or $\gamma$. To translate this bijectively,
we only have to leave unchanged a factor associated to a label $\alpha$, 
and to apply a mirror reflexion to a factor associated to a label $\gamma$.
Figure \ref{fig:bij} illustrates this correspondence.
Thanks to Proposition~ \ref{prop:bij} and Lemma~ \ref{lem:bij},
we get a bijection denoted by $\Phi$, 
between staircase tableaux and lazy paths (see Figure~ \ref{fig:from_st_to_lazy_path}).

\begin{figure}[ht]
    \centerline{
		$
		\begin{array}{c}
			\includegraphics[scale=\scaleFigure]{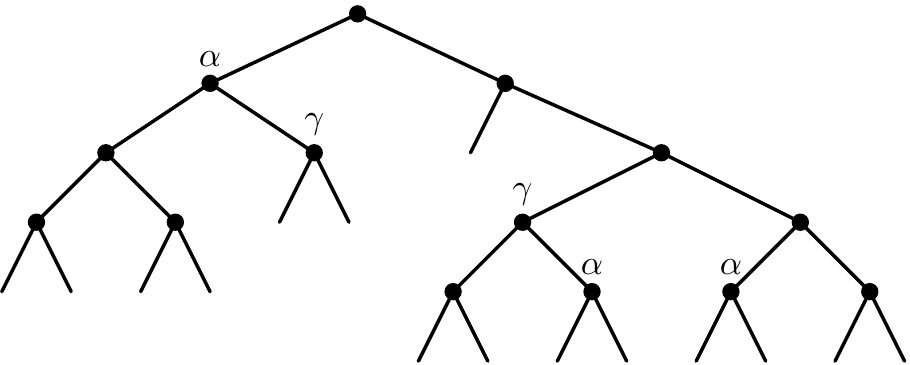}
		\end{array}
		$
	}
    \centerline{
		$
		\begin{array}{c}
			\includegraphics[scale=\scaleFigure]{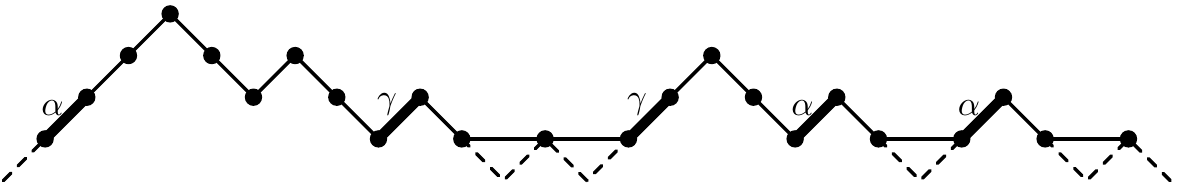}
		\end{array}
		$
	}
    \centerline{
		$
		\begin{array}{c}
			\includegraphics[scale=\scaleFigure]{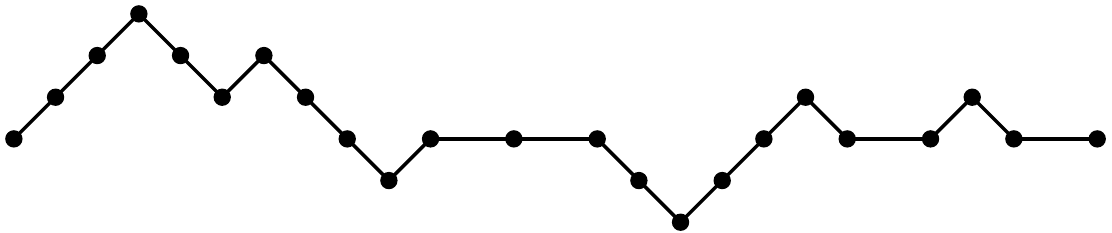}
		\end{array}
		$
	}
\caption{From labeled trees to lazy paths}
\label{fig:bij}
\end{figure}

\begin{figure}[ht]
    \centerline{
		$
		\begin{array}{c}
			\includegraphics[scale=\scaleFigure]{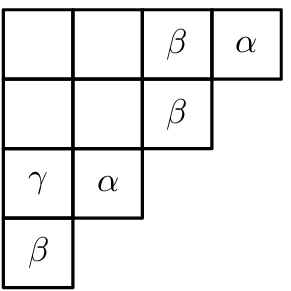}
		\end{array}
		\longleftrightarrow
		\begin{array}{c}
			\includegraphics[scale=\scaleFigure]{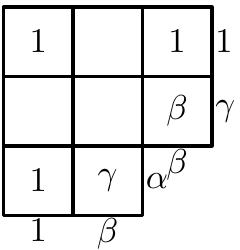}
		\end{array}
		\longleftrightarrow
		\begin{array}{c}
			\includegraphics[scale=\scaleFigure]{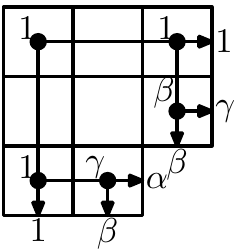}
		\end{array}
		\longleftrightarrow
		$
	}
    \centerline{
		$
		\begin{array}{c}
			\includegraphics[scale=\scaleFigure]{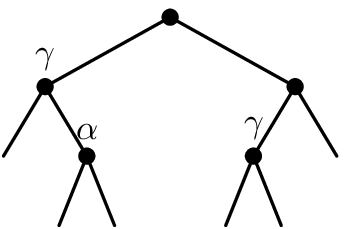}
		\end{array}
		\longleftrightarrow
		\begin{array}{c}
			\includegraphics[scale=\scaleFigure]{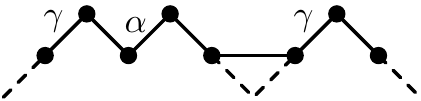}
		\end{array}
		\longleftrightarrow
		\begin{array}{c}
			\includegraphics[scale=\scaleFigure]{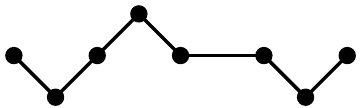}
		\end{array}
		$
	}
\caption{Bijection $\Phi$}
\label{fig:from_st_to_lazy_path}
\end{figure}

We recall the following definition from \cite{CW2}.
A row \emph{indexed} by $\beta$ or $\delta$ 
in a staircase tableau $\T$ is a row such that its left-most label is  $\beta$ or $\delta$.
In the same way, a column indexed by an $\alpha$ or a $\gamma$ 
is a column such that its top-most label is  $\alpha$ or $\gamma$.
For example, the staircase tableau on the left of Figure~\ref{fig:ST} has 
2 columns indexed by $\alpha$ and 1 row indexed by $\delta$.

\begin{prop}
The application $\Phi$ defines a bijection from the set of staircase tableaux $\T$ without label $\delta$ and weight $q$,
of size $n$, to the set of lazy paths of size $n$. Moreover, if we denote:
$D(\Phi(\T))$ the number of $(1,-1)$ steps, 
$M(\Phi(\T))$ the number of $(1,1)$ steps,
$A(\Phi(\T))$ the length of the initial maximal sequence of $(1,1)$ steps,
$H(\Phi(\T))$ the number of $(2,0)$ steps,
$F(\Phi(\T))$ the number of factors,
and
$N(\Phi(\T))$ the number of negative factors
in $\Phi(\T)$, then:
\begin{itemize}
\item the number of $\gamma$ labels in $\T$ is given by $N(\Phi(\T))$;
\item the number of $\alpha$ labels in $\T$ is given by $M(\Phi(\T))-N(\Phi(\T))$;
\item the number of $\beta$ labels in $\T$ is given by $D(\Phi(\T))+H(\Phi(\T))-A(\Phi(\T))$;
\item the number of columns indexed by $\beta$ in $\T$ is given by  $H(\Phi(\T))$;
\item the number of rows indexed by $\alpha$ or $\gamma$ in $\T$ is given by  $A(\Phi(\T))$.
\end{itemize}
\end{prop}

\begin{proof}
We still have to check the assertions about the different statistics.
We recall that, as  defined in \cite{CW2}, 
rows indexed by $\beta$ or $\delta$ (resp. columns indexed by $\alpha$ or $\gamma$) in a staircase tabelau $\T$
correspond to non-root $1$ labels in the first column (resp. first row) of $\Lambda(\T)$. 
We have:
\begin{itemize}
\item the number of $\gamma$ labels is by definition the number of negative factors in $N(\Phi(\T))$;
\item the number of $\alpha$ or $\gamma$ labels is $M(\Phi(\T))$;
\item the number of $\beta$ labels is the number of external nodes minus the number of nodes in the left branch of the tree, thus $D(\Phi(\T))+H(\Phi(\T))-A(\Phi(\T))$;
\item columns indexed by $\beta$ correspond to  nodes in the right branch of the tree, their number is $H(\Phi(\T))$;
\item rows indexed by $\alpha$ or $\gamma$ correspond to  nodes in the left branch of the tree, their number is $A(\Phi(\T))$.
\end{itemize}
\end{proof}

\begin{rem}
\rm We may observe that among this class of staircase tableaux $\T$, those who have only 
 $\beta$ or $\gamma$ labels on the diagonal, {\it i.e.} such that $t(\T)=0$ are in bijection with binary trees
 whose internal nodes are of left depth at most $1$, and such
that we forbid the $\alpha$ label on external nodes.
The bijection $\Phi$ sends these tableaux onto lazy paths
of height and depth bounded by $1$ and whose factors preceding either a $(2,0)$ step or the end of the path are negative 
({\it cf.} Figure~ \ref{fig:frob_path}).
Let us denote by $S(n)$ the number of such path of size $n$. 
By decomposing the path with respect to its first two factors, we may write 
$$S(n)=3S(n-1)-S(n-2)$$ 
which corresponds ({\it cf.} the entry {\tt A001519} in \cite{oeis}) to the recurrence of odd Fibonacci numbers
$F_{2n+1}$, as claimed in Corollary 3.10 of \cite{CSSW}.
\end{rem}
\begin{figure}[ht]
    \centerline{
		$
		\begin{array}{c}
			\includegraphics[scale=\scaleFigure]{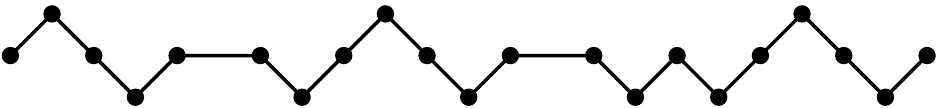}
		\end{array}
		$
	}
\caption{An odd Frobenius path}
\label{fig:frob_path}
\end{figure}

\begin{rem}
\rm Another interesting special case concerns staircase tableaux of size $n$ without any $\delta$ or $\gamma$ labels,
and without weight $q$. It is obvious that the bijection $\Phi$ maps these tableaux onto binary trees,
enumerated by Catalan numbers  $C_n=\frac 1{n+1}{2n \choose n}$. 
Moreover, if we keep track of the number  $k$ of external nodes labeled with $\beta$, 
we get a bijection with Dyck paths $\pi$ of size  $n$ with exactly $k$ peaks,
enumerated by Narayana numbers $N(n,k)$. 
\end{rem}

\bigskip
\noindent
{\bf\large Forthcoming objective.}
We are convinced that LTLTs, because of their insertion algorithm, are objects 
that are both natural and easy to use, as shown in this paper on some special cases. 
Since a nice feature of our insertion procedure is to work on the boundary edges,
which encode the states in the PASEP,
an objective is to use these objects to describe combinatorially the 
general case of the PASEP model \cite{CW2}.
To do that, we have to find an alternate description of the weight on LTLTs, hopefully simpler than
the one defined on staircase tableaux  in Definition~\ref{weight}.

\bigskip
\noindent
{\bf\large Acknowledgements}
\label{sec:ack}
This research was driven by computer exploration using the open-source mathematical software \texttt{Sage} \cite{sage} and its algebraic combinatorics features developed by the \texttt{Sage-Combinat} community~\cite{Sage-Combinat}.


\end{document}